\documentclass[reqno,b5paper]{amsart}
\usepackage{amsmath}
\usepackage{amssymb}
\usepackage{amsthm}
\usepackage{enumerate}
\usepackage[mathscr]{eucal}
\theoremstyle{plain}
\newtheorem{thm}{Theorem}[section]

\newtheorem{Example}{Example}[section]

\theoremstyle{definition}
\newtheorem{defn}{Definition}[section]
\newtheorem{rem}{Remark}[section]

\begin{document}

\setcounter {page}{1}
\title{ Statistical convergence in metric-like spaces }

\author[ P. Malik, S. Das ]{Prasanta Malik*, and Saikat Das*\ }
\newcommand{\acr}{\newline\indent}
\maketitle
\address{{*\,} Department of Mathematics, The University of Burdwan, Golapbag, Burdwan-713104,
West Bengal, India.
                Email: pmjupm@yahoo.co.in, dassaikatsayhi@gmail.com \acr
           }

\maketitle


\begin{abstract}
In this paper we introduce the notions of statistical
convergence and statistical Cauchyness of sequences in a metric-like space.
We study some basic properties of these notions.
\end{abstract}

\textbf{keywords :} natural density, statistical convergence, statistically Cauchy, metric-like space.


\section{\textbf{ Introduction and Background }}

The notion of convergence of sequences of real numbers was
extended to the notion of statistical convergence by Fast
\cite{Fa} (and also independently by Schoenberg \cite{Sc}), which
is primarily based on the concept of the natural density of the
subsets of the set of all natural numbers. any studies have been
done on this concept ( see \cite{Fr1, Fr2, Ko, Sl}).

On the other hand the concept of partial metric space was first
introduced by Matthews \cite{Matt}, as a generalization of the
usual notion of metric space. In \cite{Am}, A. A. Harandi
introduced the concept of metric-like space which is a
generalization of the concepts of metric space as well as partial
metric space and studied the notions of convergence and Cauchyness
of sequences in a metric-like space.

In this paper we introduce and study the notion of statistical
convergence of sequences in a metric-like space. Also introducing
the notion of statistical Cauchyness in a metric-like space we
have examined its relationship with statistical convergence.


\section{\textbf{ Basic Definitions and Notations }}

In this section we recall some basic definitions and notations which will be needed in our study.
Throughout the paper, $\mathbb{R}_{\geq 0}$ denotes the set of all non-negative real numbers.

\begin{defn} \cite{Am}: \label{ defn2}
Let $\mathcal X$ be a non-empty set and a mapping $\delta: \mathcal{ X \times X}
\rightarrow \mathbb{R}_{\geq 0}$  is said to be a metric-like on $\mathcal X$ if for
any $x, y, z \in \mathcal X$, the following conditions are satisfied :
\begin{itemize}
\item[$(\delta1)$] $\delta(x,y)= 0 \Rightarrow x=y;$
\item[$(\delta2)$] $\delta(x,y)= \delta(y,x);$
\item[$(\delta3)$] $\delta(x,z)\leq \delta(x,y) + \delta(y,z)$.
\end{itemize}

The pair $(\mathcal X,\delta)$ is then called a \textit{metric-like space}.
Throughout the paper, $(\mathcal X,\delta)$ will denote a metric-like space, unless
otherwise mentioned.
\end{defn}

\begin{rem}
We see that every \textit{metric space}
is a \textit{partial metric space} and that of every \textit{partial metric space} is a \textit{metric-like
space}, but the converse are not true.
\end{rem}

\begin{defn} \cite{Am}:
Let $x_{o}$ be a point in a metric-like space $(\mathcal X,\delta)$ and let $\epsilon > 0$.
Then the open $\delta$-ball with centered at $x_{o}$ and radius $\epsilon > 0$ in $(\mathcal X,\delta)$
is denoted by $\mathcal{B}_{\delta}(x_{o};\epsilon)$ and is defined by
\begin{align*}
\mathcal{B}_{\delta}(x_{o};\epsilon)= \{ x \in \mathcal{X}: \left|\delta(x,x_{o})-
\delta(x_{o},x_{o})\right|< \epsilon \}.
\end{align*}
\end{defn}

\begin{defn} \cite{Am}:
A sequence $\{x_{n}\}_{n\in \mathbb N}$ in a metric-like space $(\mathcal X,\delta)$ is said to be
convergent to a point $x_{o} (\in \mathcal X)$, if for every $\epsilon > 0$, there exists $k_{o}
\in \mathbb N$ such that
\begin{align*}
&\left|\delta(x_{n}, x_{o})- \delta(x_{o}, x_{o})\right|< \epsilon,  &\forall~ n\geq k_{o}\\
\mbox{i.e.}~ & x_{n} \in \mathcal{B}_{\delta}(x_{o};\epsilon),  &\forall~ n\geq k_{o}.
\end{align*}
In this case, we write $\displaystyle{\lim_{n\rightarrow \infty}}x_{n} = x_{o}$.
\end{defn}

\begin{defn} \cite{Am}
A sequence $\{x_{n}\}_{n\in \mathbb{N}}$ in a metric-like space $(\mathcal{X},\delta)$ is said to be Cauchy
(or $\delta$-Cauchy) if there exists $l \geq 0$ such that $\displaystyle{\lim_{m ,n\rightarrow \infty}}
\delta(x_{m},x_{n})= l$, i.e. for every $\epsilon > 0 ~\exists~ k_{o} \in \mathbb{N}$ such that
\begin{align*}
\left|\delta(x_{m},x_{n}) - l\right| < \epsilon, ~\forall~ m, n \geq k_{o}.
\end{align*}
\end{defn}

\begin{defn} \cite{Fr1}:
Let $\mathcal P$ be a subset of $\mathbb N$. For each $n \in
\mathbb{N}$, let $\mathcal{P}(n)$ denote the cardinality of the
set $\{k \leq n: k \in \mathcal{P}\}$. We say that the set
$\mathcal{P}$ has natural density $d(\mathcal{P})$ if the limit
$\displaystyle{\lim_{n\rightarrow
\infty}}\frac{\mathcal{P}(n)}{n}$ exists finitely and
\begin{align*}
d(\mathcal{P}) = \displaystyle{\lim_{n\rightarrow \infty}} \frac{\mathcal{P}(n)}{n}.
\end{align*}
\end{defn}

\begin{defn} \cite{Fr1}:
Let $\{x_{n}\}_{n\in \mathbb N}$ be a sequence of real numbers. Then $\{x_{n}\}_{n\in \mathbb N}$
is said to be statistically convergent to $x_{o} (\in \mathbb R)$, if for any $\epsilon > 0,~ d(\mathcal
A(\epsilon ))= 0$, where
\begin{align*}
\mathcal A(\epsilon)= \left\{n\in \mathbb N: \left|x_{n} - x_{o}\right|\geq \epsilon\right\}.
\end{align*}
In this case we write $st-\displaystyle{\lim_{n\rightarrow \infty}}x_{n}= x_{o}.$
\end{defn}


\section{\textbf{ Statistical convergence in a metric-like space}}

\begin{defn}: A sequence $\{x_{n}\}_{n\in \mathbb N}$ in a metric-like space ($\mathcal{X}, \delta$)
is said to be statistically convergent to a point $x_{o} (\in \mathcal X)$ if for
every $\epsilon > 0,~ d(\mathcal{A}(\epsilon))= 0$, where
\begin{align*}
\mathcal{A}(\epsilon) &= \{n \in \mathbb{N}: \left| \delta(x_{n},x_{o})- \delta(x_{o},x_{o})
\right|\geq \epsilon \} \\
& = \{n\in \mathbb{N}: x_{n}\in \mathcal{B}_{\delta}(x_{o},\epsilon)\}
\end{align*}
i.e. if $st-\displaystyle{\lim_{n\rightarrow \infty}} \delta(x_{n},x_{o}) = \delta(x_{o},x_{o})$.
In this case, we write $st-\displaystyle{\lim_{n\rightarrow \infty}} x_{n}= x_{o}$
\end{defn}

Limit of a statistically convergent sequence in a metric-like space may not be unique.

\begin{defn}
Let $(\mathcal{X},\delta)$ be a metric-like space. A sequence $\{x_{n}\}_{n \in \mathbb N}$ in $\mathcal{X}$
is said to be bounded in $\mathcal{X}$ if there exists $M > 0$ such that
\begin{align*}
\delta(x_{m},x_{n}) < M, ~\forall~ m, n \in \mathbb{N}.
\end{align*}
\end{defn}

\begin{thm} \label{ Theo 1}
A convergent sequence in a metric-like space $(\mathcal{X},\delta)$ is bounded.
\end{thm}

\begin{proof}
Proof is trivial so omitted.
\end{proof}

We now site an example of a sequence in a metric-like space which is statistically convergent
but not usually convergent.

\begin{Example}
Let $\mathcal{X} = \mathbb{R}_{\geq 0}$ and a mapping $\delta : \mathcal{X} \times \mathcal{X}\rightarrow
 \mathbb{R}_{\geq 0}$ be defined by
\begin{align*}
\delta(x,y) = \begin{cases}
0, &\mbox{if}~ x = y \\
x + y, &\mbox{otherwise}.
\end{cases}
\end{align*}
Then $(\mathcal{X},\delta)$ is a metric-like space. Let us consider a sequence $\{x_{n}\}_{n\in \mathbb{N}}$
in $\mathcal{X}$ be defined as follows:
\begin{align*}
x_{n} = \begin{cases}
k, &\mbox{if}~ n = k^{2} ~(k = 1, 2,\ldots) \\
2, &\mbox{otherwise}.
\end{cases}
\end{align*}
Let $\epsilon > 0$ be given. Then
\begin{align*}
\delta(x_{n},2) = \begin{cases}
k + 2, &\mbox{if}~ n = k^{2} ~(k\in \mathbb{N}, k \neq 2) \\
0, &\mbox{otherwise}.
\end{cases}
\end{align*}
Therefore $\{n\in \mathbb{N}: \left|\delta(x_{n},2) - \delta(2,2)\right|\geq \epsilon\} \subset
\{1^{2}, 2^{2}, \ldots\}$. As $d(\{1^{2}, 2^{2}, \ldots\}) = 0$, therefore $d(\{n\in \mathbb{N}: \left|
\delta(x_{n},2) - \delta(2,2)\right| \geq \epsilon\}) = 0$. Thus
\begin{align*}
st-\displaystyle{\lim_{n\rightarrow \infty}} \delta(x_{n},2) = \delta(2,2),
~\mbox{and so}~ st-\displaystyle{\lim_{n\rightarrow \infty}} x_{n} = 2.
\end{align*}
Let $M > 0$. Then there exists $k\in \mathbb{N}$ such that $k \leq M < (k + 1)$. Therefore
\begin{align*}
\delta(x_{k^{2}},x_{1}) = k + 1 > M.
\end{align*}
So, the sequence $\{x_{n}\}_{n\in \mathbb{N}}$ is unbounded in $(\mathcal{X},\delta)$ and hence using
Theorem \ref{ Theo 1} it is not usually convergent in $(\mathcal{X},\delta)$.
\end{Example}

Thus the notion of statistical convergence in a metric-like space is a natural generalization of the
usual notion of convergence of sequences.

\begin{thm} \label{ Theo 3}
Let $(\mathcal{X},\delta)$ be a metric-like space and $\{x_{n}\}_{n\in \mathbb{N}}$ be a sequence in
$\mathcal{X}$ and $x_{o} \in \mathcal{X}$. Then $st-\displaystyle {\lim_{n\rightarrow \infty}} x_{n} = x_{o}$
if and only if there exists a subsequence $\{x_{n_{k}}\}_{k\in \mathbb{N}}$ of $\{x_{n}\}_{n\in \mathbb{N}}$ such that
$d(\{n_{1}, n_{2}, n_{3}, \ldots\}) = 1$ and
$\displaystyle{\lim_{k\rightarrow \infty}} x_{n_{k}} = x_{o}$.
\end{thm}

\begin{proof}
First let there exists a subsequence $\{x_{n_{k}}\}_{k\in \mathbb{N}}$ of $\{x_{n}\}_{n\in \mathbb{N}}$ such that
$d(\{n_{1},\\ n_{2}, n_{3}, \dots\})$ = 1
and $\displaystyle{\lim_{k\rightarrow \infty}} x_{n_{k}} = x_{o}$.

Let $\epsilon > 0$ be given. Then there exists $k_{o}\in \mathbb{N}$ such that
\begin{align} \label{ eq 3}
\left|\delta(x_{n_{k}},x_{o}) - \delta(x_{o},x_{o})\right| < \epsilon, ~\forall~ k \geq k_{o}.
\end{align}
Let $\mathcal{A}(\epsilon) = \{n\in \mathbb{N}: \left|\delta(x_{n},x_{o}) - \delta(x_{o},x_{o})\right| \geq \epsilon\}$.
Then from (\ref{ eq 3}) we have $\mathcal{A}(\epsilon) \subset \mathbb{N} - \{n_{k}: k \geq k_{o}\}$.
Since $d(\{n_{k}: k \geq k_{o}\}) = 1$ therefore
$d(\mathbb{N} - \{n_{k}: k \geq k_{o}\}) = 0$ and hence $d(\mathcal{A}(\epsilon)) = 0$. Thus,
\begin{align*}
st-\displaystyle{\lim_{n\rightarrow \infty}} x_{n} = x_{o}.
\end{align*}

Conversely, let $st-\displaystyle{\lim_{n\rightarrow \infty}} x_{n} = x_{o}$.
For any $j\in \mathbb{N}$ we set $\mathcal{P}_{j} = \{n\in \mathbb{N}: \left|\delta(x_{n},x_{o}) -
\delta(x_{o},x_{o})\right| < \frac{1}{j}\}$. Since $st-\displaystyle
{\lim_{n\rightarrow \infty}}x_{n} = x_{o}$ we have
$d(\mathcal{P}_{j}) = 1, ~\forall~ j\in \mathbb{N}$.
Therefore each $\mathcal{P}_{j}$ is an infinite subset of $\mathbb{N}$ and
$\mathcal{P}_{1} \supset \mathcal{P}_{2} \supset \ldots$.

As $d(\mathcal{P}_{1}) = 1$, so $\left|\{m \leq n: m\in \mathcal{P}_{1}\}\right|/n > 0, ~\forall~
n \geq v_{1}$, where $v_{1}$ is the least element of $\mathcal{P}_{1}$.
Again, since $d(\mathcal{P}_{2}) = 1, ~\exists~ v_{2}\in \mathcal{P}_{2}$ with $v_{2} > v_{1}$ such that
\begin{align*}
\frac{\left|\{m \leq n: m\in \mathcal{P}_{2}\}\right|}{n} > \frac{1}{2}, ~\forall~ n \geq v_{2}.
\end{align*}
Similarly, there exists $v_{3}\in \mathcal{P}_{3}$ with $v_{3} > v_{2}$ such that
\begin{align*}
\frac{\left|\{m \leq n: m\in \mathcal{P}_{3}\}\right|}{n} > \frac{2}{3}, ~\forall~ n \geq v_{3}.
\end{align*}
Proceeding in this way we get a sequence $\{v_{1}< v_{2}< v_{3} < \ldots\}$ of natural numbers such that
\begin{align} \label{ eq 4}
\frac{\left|\{m \leq n: m\in \mathcal{P}_{j}\}\right|}{n} > \frac{j-1}{j}, ~\forall~ n \geq v_{j}.
\end{align}
Let $K = (\left[1, v_{1}\right] \cap \mathbb{N}) \cup (\displaystyle{\bigcup^{\infty}_{j = 1}}\left[v_{j}, v_{j+1}\right] \cap
\mathcal{P}_{j})$. Then $K \subset \mathbb{N}$ and using (\ref{ eq 4}) we have for any $n \in \mathbb{N}$ with
$v_{j} \leq n < v_{j+1}$,
\begin{align*}
\frac{\left|\{m \leq n: m\in K\}\right|}{n} \geq \frac{\left|\{m \leq n: m\in \mathcal{P}_{j}\}\right|}{n} > \frac{j-1}{j}.
\end{align*}
Therefore, $\displaystyle{\lim_{n\rightarrow \infty}} \frac{\left|\{m \leq n: m\in K\}\right|}{n} = 1$ i.e. $d(K) = 1$.

Let $\epsilon > 0$ be given. Then $\exists~ l\in \mathbb{N}$ such that $\frac{1}{l} < \epsilon$. Choose $n\in K$ with $n \geq v_{l}$.
As $n \in K, ~\exists~ j \in \mathbb{N}$ with $j \geq l$ such that $n \in \left[v_{j}, v_{j+1}\right] \cap \mathcal{P}_{j}$. Then
\begin{align*}
& \left|\delta(x_{n},x_{o}) - \delta(x_{o},x_{o})\right| < \frac{1}{j} \leq \frac{1}{l} < \epsilon. \\
\mbox{Thus,}~ & \left|\delta(x_{n},x_{o}) - \delta(x_{o},x_{o})\right| < \epsilon, ~\forall~ n\in K ~\mbox{with}~ n \geq v_{l}.
\end{align*}
This implies,
$\displaystyle{\lim_{n\rightarrow \infty \atop n\in K}} \delta(x_{n},x_{o}) = \delta(x_{o},x_{o})$. Therefore
writing $K = \{n_{1}, n_{2}, n_{3},\ldots\}$ we have $d(K) = 1$ and $\{x_{n_{k}}\}_{k\in \mathbb{N}}$ is a
subsequence of $\{x_{n}\}_{n\in \mathbb{N}}$ such that
$\displaystyle{\lim_{k\rightarrow \infty}} x_{n_{k}} = x_{o}$.
\end{proof}

\begin{thm} \label{ Theo 4}
Let $\{x_{n}\}_{n\in \mathbb{N}}$ be a sequence in a metric-like space $(\mathcal{X},\delta)$ such that
$st-\displaystyle{\lim_{n\rightarrow \infty}} x_{n} = x_{o} ~(\in \mathcal{X})$ and $\delta(x_{o},x_{o}) = 0$.
Then $st-\displaystyle{\lim_{n\rightarrow \infty}} \delta(x_{n},y) = \delta(x_{o},y)$ for all $y\in \mathcal{X}$.
\end{thm}

\begin{proof}
As $st-\displaystyle{\lim_{n\rightarrow \infty}} x_{n} = x_{o}$ and $\delta(x_{o},x_{o}) = 0$ therefore $st-\displaystyle
{\lim_{n\rightarrow \infty}} \delta(x_{n},x_{o}) = 0$. Now let $\epsilon > 0$ be given. Then $d(\{n\in \mathbb{N}: \delta(x_{n},
x_{o}) < \epsilon\}) = 1$. We set, $\mathcal{A}(\epsilon) = \{n\in \mathbb{N}: \delta(x_{n},x_{o}) < \epsilon\}$. Let $y \in
\mathcal{X}$ and $n \in \mathbb{N}$. Then,
\begin{align} \label{ eq 5}
& \delta(x_{n},y) \leq \delta(x_{n},x_{o}) + \delta(x_{o},y) \nonumber \\
\Rightarrow~ & \delta(x_{n},y) - \delta(x_{o},y) \leq \delta(x_{n},x_{o}).
\end{align}
Also,
\begin{align} \label{ eq 6}
& \delta(x_{o},y) \leq \delta(x_{o},x_{n}) + \delta(x_{n},y) \nonumber \\
\Rightarrow~ & - \delta(x_{n},x_{o}) \leq \delta(x_{n},y) - \delta(x_{o},y).
\end{align}
Then from (\ref{ eq 5}) and (\ref{ eq 6}) we have $\left|\delta(x_{n},y) - \delta(x_{o},y)\right|
\leq \delta(x_{n},x_{o})$. Therefore for all $n \in \mathcal{A}(\epsilon)$,
\begin{align*}
\left|\delta(x_{n},y) - \delta(x_{o},y)\right| \leq \delta(x_{n},x_{o}) < \epsilon.
\end{align*}
This implies, $\mathcal{A}(\epsilon) \subset \{n\in \mathbb{N}: \left|\delta(x_{n},y) - \delta(x_{o},y)\right| < \epsilon \}$.
As $d(\mathcal{A(\epsilon)}) = 1$, so $d(\{n\in \mathbb{N}: \left|\delta(x_{n},y) - \delta(x_{o},y)\right| < \epsilon\}) = 1$.
Hence, $st-\displaystyle{\lim_{n\rightarrow \infty}} \delta(x_{n},y) = \delta(x_{o},y), ~\forall~ y\in \mathcal{X}$.
\end{proof}


\section{\textbf{Statistical Cauchyness in a metric-like space}}

In this section using the notion of double natural density ( see \cite{Mur}) we introduce the
notion of statistical Cauchy sequences in a metric-like space.

\begin{defn} \cite{Pr}:
A double sequence $\{x_{jk}\}_{j,k\in \mathbb{N}}$ of real numbers is said to converge to a
real number $l$ in \textit{Pringsheim's sense}, if for every $\epsilon > 0, ~\exists~ k_{o}\in
\mathbb{N}$ such that
\begin{align*}
\left|x_{jk} - l\right| < \epsilon, ~\forall~ j, k \geq k_{o}.
\end{align*}
In this case we say $l$ is a limit of the sequence $\{x_{jk}\}_{j, k\in \mathbb{N}}$.
\end{defn}

\begin{defn} \cite{Mur}
Let $K \subset \mathbb{N \times N}$ and $K(m,n)$ denote the cardinality of the set
$\{(i,j) \in K: i \leq m, j \leq n\}$. We say that the set $K$ has double natural density
$d_{2}(K)$, if the sequence $\{\frac{K(m,n)}{mn}\}$ has a limit in \textit{Pringsheim's sense},
and
\begin{align*}
d_{2}(K) = \displaystyle{\lim_{m,n \rightarrow \infty}} \frac{K(m,n)}{mn}.
\end{align*}
\end{defn}

\begin{defn}
A sequence $\{x_{n}\}_{n \in \mathbb{N}}$ in a metric-like space $(\mathcal{X},\delta)$ is said to be statistically
Cauchy if there exists $l \geq 0$ such that $st-\displaystyle{\lim_{m, n\rightarrow \infty}} \delta(x_{m}
,x_{n}) = l$, i.e. for every $\epsilon > 0, ~d_{2}(\mathcal{A}(\epsilon)) = 0$, where
\begin{align*}
\mathcal{A}(\epsilon) = \{(m, n)\in \mathbb{N \times N}: \left|\delta(x_{m},x_{n}) - l\right| \geq \epsilon\}.
\end{align*}
\end{defn}

\begin{rem}
In a metric-like space statistically Cauchy sequences may not be statistically convergent.
To show this we consider the following example.
\end{rem}

\begin{Example}
Let $\mathcal{X}$ denotes the set of all positive real numbers and $\delta: \mathcal{X \times X} \rightarrow
\mathbb{R}_{\geq 0}$ be defined as follows:
\begin{align*}
\delta(x,y) = \begin{cases}
0, & ~\mbox{if}~ x = y ~\mbox{and}~ x ~\mbox{is irrational} \\
x + y, & ~\mbox{otherwise}.
\end{cases}
\end{align*}
Then $(\mathcal{X},\delta)$ is a metric-like space but not a metric space. Now we
consider a sequence $\{x_{n}\}_{n\in \mathbb{N}}$ in $\mathcal{X}$, be defined as follows:
\begin{align*}
x_{n} = \begin{cases}
k, & \mbox{if}~ n = k^{2} ~(k = 1, 2,\ldots) \\
\frac{1}{n}, & \mbox{otherwise}.
\end{cases}
\end{align*}
Then,
\begin{align*}
\delta(x_{m},x_{n}) = \begin{cases}
k_{1} + k_{2}, & \mbox{if}~ m = k^{2}_{1}, n = k^{2}_{2} ~\mbox{ for some}~ k_{1}, k_{2} \in \mathbb{N} \\
k + \frac{1}{n}, & \mbox{if}~ m = k^{2}, ~\mbox{for some}~ k\in \mathbb{N} ~\mbox{and n is not a perfect square} \\
\frac{1}{m} + \frac{1}{n}, & \mbox{ if none of m, n is a perfect square}.
\end{cases}
\end{align*}
Let $\epsilon > 0$ be given. Then there exists $k_{o} \in \mathbb{N}$ such that
\begin{align*}
\frac{1}{m} + \frac{1}{n} < \epsilon, ~\forall~ m, n \geq k_{o}.
\end{align*}
This implies,
\begin{align*}
\{(m,n)\in \mathbb{N \times N} : \left|\delta(x_{m},x_{n}) - 0\right| \geq \epsilon\} &\subset (\mathcal{P \times P})
 \cup (\mathcal{P} \times \mathcal{P}^{c})\\
 & \cup \{(m,n)\in \mathbb{N \times N}: m < k_{o}~or~ n < k_{o}\},
\end{align*}
where $\mathcal{P}$ denotes the set of all perfect square natural numbers and $\mathcal{P}^{c}$ denotes it's
complement in $\mathbb{N}$.

Since $d_{2}(\mathcal{P \times P}) = 0, ~d_{2}(\mathcal{P} \times \mathcal{P}^{c}) = 0$ and
$d_{2}(\{(m,n)\in \mathbb{N \times N}: m < k_{o} ~or~ n < k_{o}\}) = 0$ so,
\begin{align*}
& d_{2}(\{(m,n)\in \mathbb{N \times N}: \left|\delta(x_{m},x_{n}) - 0\right|\geq \epsilon\}) = 0 \\
\Rightarrow~ & st-\displaystyle{\lim_{m,n \rightarrow \infty}} \delta(x_{m},x_{n}) = 0.
\end{align*}
This implies that the sequence $\{x_{n}\}_{n\in \mathbb{N}}$ is statistically Cauchy in $\mathcal{X}$.
But the sequence is not statistically convergent to any point of $\mathcal{X}$.
\end{Example}


{\bf Acknowledgement : } The second author is grateful to the University Grants Commission, India for
financial support under UGC-JRF scheme during the preparation of this paper.

\end{document}